\documentclass[11pt]{article}
\usepackage{xcolor}
\usepackage{amsthm}
\usepackage{amsfonts}
\usepackage{graphicx}
\usepackage{subfig}
\usepackage{latexsym}
\usepackage{amssymb}
\usepackage{url}
\usepackage{bbm}
\usepackage{mathtools}
\mathtoolsset{showonlyrefs}

\newcommand{\CC}{\mathbb{C}}
\newcommand{\R}{\mathbb{R}}

\newcommand{\E}{\mathbb{E}}
\newcommand{\N}{\mathbb{N}}

\newcommand{\1}{\mathbbm{1}}
\renewcommand{\P}{\mathbb P}

\newcommand{\al}{\alpha}
\newcommand{\be}{\beta}
\newcommand{\la}{\lambda}

\newcommand{\ga}{\gamma}

\newcommand{\ep}{\varepsilon}

\newcommand{\de}{\delta}
\newcommand{\te}{\theta}
\newcommand{\vt}{\vartheta}

\newcommand{\om}{\omega}

\newcommand{\vp}{\varphi}

\newcommand{\f}{\mathcal F}

\newcommand{\proba}{(\Omega ,\mathcal{F},(\f_t)_{t\geq0},\P)}

\newcommand{\bee}{\begin{equation}}
\newcommand{\eee}{\end{equation}}
\newcommand{\beea}{\begin{array}}
\newcommand{\eeea}{\end{array}}

\renewcommand{\theequation}{\arabic{section}.\arabic{equation}}

\theoremstyle{plain}
\newtheorem{prop}{Proposition}[section]
\newtheorem{cor}[prop]{Corollary}
\newtheorem{theo}[prop]{Theorem}
\newtheorem{lem}[prop]{Lemma}

\theoremstyle{definition} 

\newtheorem{ex}[prop]{Example}

\newtheorem{rem}[prop]{Remark}

\begin{document}

\title{Semiparametric estimation of McKean-Vlasov SDEs\thanks{The authors
gratefully acknowledge financial support of ERC Consolidator Grant 815703
``STAMFORD: Statistical Methods for High Dimensional Diffusions''.}} 
\author{Denis Belomestny\thanks{Faculty
of Mathematics, University of Duisburg-Essen, 
E-mail: denis.belomestny@uni-due.de.}  \and
Vytaut\.e Pilipauskait\.e\thanks{Department
of Mathematics, University of Luxembourg, 
E-mail: vytaute.pilipauskaite@uni.lu.} \and
Mark Podolskij\thanks{Department
of Mathematics, University of Luxembourg,
E-mail: mark.podolskij@uni.lu.}}

\maketitle

\begin{abstract}
In this paper we study the problem of semiparametric estimation for a class of McKean-Vlasov stochastic differential equations. Our aim is to estimate the drift coefficient of  a MV-SDE based on observations of the corresponding particle system. We propose a semiparametric estimation procedure and derive the rates of convergence for the resulting estimator.  We further prove that the obtained rates are essentially optimal in the minimax sense. 

\bigskip
{\it Key words}: deconvolution, McKean-Vlasov SDEs, mean-field models, minimax bounds, semiparametric estimation.\

\bigskip

{\it AMS 2010 subject classifications.} 62G20, 62M05, 60G07, 60H10.

\end{abstract}

\section{Introduction} \label{sec1}
\setcounter{equation}{0}
\renewcommand{\theequation}{\thesection.\arabic{equation}}

In the past fifty years  diffusion processes found numerous applications in natural and social sciences, and a variety of statistical methods have been investigated in the setting of SDEs during the last few decades. Maximum likelihood estimation and Bayesian approach are the most well established parametric methods in the literature; we refer to the monograph \cite{kutoyants2013statistical}. 
When the likelihood function is not available in a closed form, quasi likelihood methods provide an alternative approach to parameter estimation, see \cite{chang2011approximate} and references therein. The most recent contributions to nonparametric inference for diffusions are \cite{nickl2017nonparametric,strauch2018adaptive}.
These belong to the most successful tools when analysing 
estimation problems in different observation schemes of a diffusion.

Many diffusion models in natural and applied sciences can be viewed as continuous-time processes with complex and nonlinear probabilistic structure. For example, in statistical mechanics nonlinear diffusion models and particle systems have long and successful history. In general, nonlinear Markov processes are stochastic processes whose transition functions may depend not only on the current state of the process but also on the
current distribution of the process. These processes were introduced by
McKean~\cite{mckean1966class} to model plasma dynamics. Later nonlinear Markov
processes were studied by a number of authors; we mention here the books of
Kolokoltsov~\cite{kolokoltsov2010nonlinear} and
Sznitman~\cite{sznitman1991topics}. These processes arise naturally in the
study of the limit behavior of a large number of weakly interacting particles and have a wide range of applications, including financial
mathematics, population dynamics, and neuroscience (see, e.g.,
\cite{frank2005nonlinear} and the references therein). In this context the mean field theory has been employed to bridge the interaction of particles at the microscopic scale and the mesoscopic features of the system. From a probabilistic point of view propagation of chaos, fluctuation analysis, and large deviations have been investigated for a variety of mean field models and nonlinear diffusions. 
\par
In recent years, one witnessed a growing interest in statistical problems for high dimensional diffusions
in general and McKean-Vlasov (MV) SDEs in particular. Statistical inference for high dimensional Ornstein-Uhlenbeck models have been investigated in \cite{CMP20,GM19}. 
A parametric problem of estimating the coefficients of a MV-SDE under a small noise assumption have been  studied in \cite{GL20,ren2019least}.
Our current work is mostly related  to  a recent paper  \cite{dellamaestra2021nonparametric}, where  based on observation of a trajectory of an interacting particle system over a fixed time horizon,  the authors study nonparametric estimation of the solution (density) of the associated nonlinear Fokker-Planck equation, together with the drift function. The underlying statistical problem turns out to be rather challenging and   \cite{dellamaestra2021nonparametric} contains only partial results. In particular, the problem of estimating  a distribution dependent drift function of a MV-SDE from the observations of the corresponding particle system at time \(T>0\) has not been yet studied in the  literature. 
\par
In this paper we  consider the case where the drift  has a semiparametric form consisting of a polynomial part, a trigonometric part and a nonparametric interaction function convolved with an unknown marginal distribution of  the underlying MV-SDE.  The goal of this research is twofold: first to propose a kernel type estimator for the drift function based on the empirical characteristic function of the particles; and, second, to  study its properties. We derive upper bounds on \(L^2\) risk of the proposed estimator and show that these bounds are essentially optimal in minimax sense for a properly chosen functional class of drift functions. In particular, we show that the convergence rates of our estimator are logarithmic under a polynomial tail behaviour of the non-parametric part of the interaction function. Our approach is based on a rigorous analysis of the related inverse problem for the underlying stationary Fokker-Planck equation and makes use of the probabilistic properties of the model obtained in \cite{cattiaux2008probabilistic,malrieu2003convergence}. To the best of our knowledge, this is the first work containing minimax optimal procedure  for semiparametric estimation  of the coefficients of MV-SDEs from discrete observations of the corresponding particle system and hence  fills an important gap in the current literature on  statistical inference for MV-SDEs. 
\par
The structure of the paper is as follows. In Section~\ref{sec2} we introduce the main setup and  recall some basic facts about  MV-SDEs and related particle systems. In Section~\ref{sec3} we formulate our main statistical problem and describe the estimation procedure. Section~\ref{sec4} is devoted to the convergence analysis of the proposed algorithm.  In particular, we derive upper bounds on \(L^2\) risk of the suggested drift estimator.  In Section~\ref{sec5} we complete our theoretical analysis  by providing   lower bounds for the nonparametric part of the model that essentially match upper bounds obtained  in Section~\ref{sec4}. 
Conclusions and outlook are presented in Section~\ref{secOut}. All proofs are collected in Section~\ref{sec6}.


\section{The particle system model and propagation of chaos} \label{sec2}
\setcounter{equation}{0}
\renewcommand{\theequation}{\thesection.\arabic{equation}}

Throughout the paper we consider a filtered probability space $\proba$, on which all stochastic processes are defined. We focus on an $N$-dimensional system of stochastic differential equations given by
\bee \label{eq:sde-part}
X_{t}^{i,N}=X_{0}^{i}+B_{t}^{i}-\frac{1}{2N}\sum_{j=1}^{N}\int_{0}^{t}\vp' (X_{s}^{i,N}-X_{s}^{j,N} )\,ds, \qquad 1 \le i \le N, \qquad t\geq0,
\eee
where $B^{i} = (B^{i}_t)_{t \ge 0}$, $1 \le i \le N$, are independent
standard Brownian motions and $X_{0}^{i}$, $1 \le i \le N$, are i.i.d.\ random variables with distribution $\mu_0(dx)$.
Here the \textit{interaction function} $\vp'$ denotes the derivative of the function $\vp \in C^2(\R)$, which satisfies the following assumption: 
\vspace{0.3 cm}
\begin{itemize}
\item[(A)] The function $\vp$ is even (i.e. $\vp(x)=\vp(-x)$ for all $x\in \R$), strictly convex: 
\bee \label{convex}
\vp''(x) \geq \la>0, \qquad \forall x \in \R,
\eee
and locally Lipschitz with polynomial growth, that is
\bee \label{Lipschitz}
|\vp'(x)-\vp'(y)|\leq |x-y||P(x)+P(y)|, \qquad \forall x,y \in \R,
\eee 
for a polynomial $P$. 
\end{itemize}
\vspace{0.3 cm}

\noindent
The particle system \eqref{eq:sde-part} has been originally studied in \cite{benachour1998nonlinear}. However, as pointed out in \cite{malrieu2003convergence}, the asymptotic properties of the system are rather ill-behaved in terms of uniformity and long term behaviour, and it is more appropriate to consider the projected particle system
\bee \label{Yprocess}
Y_{t}^{i,N}:=X_{t}^{i,N}-\frac{1}{N}\sum_{j=1}^{N}X_{t}^{j,N},\qquad 1\le i \le N, \qquad t \ge 0.
\eee
The \textit{mean field equation}, which determines the asymptotic behaviour of the process $Y^N$ at 
\eqref{Yprocess}, is given by the $1$-dimensional McKean-Vlasov equation
\bee \label{meanfield}
\overline{X}_{t}=\overline{X}_{0}+B_{t}-\frac{1}{2}\int_{0}^{t} (\vp'\star\mu_s )(\overline{X}_{s})\,ds,\qquad t\geq0,
\eee
where $\mu_t (dx)=\P(\overline{X}_t \in dx)$ and 
\[
( \varphi'\star\mu_t )(x)=\int_{\R} \varphi'(x-y)\mu_t (dy), \qquad x \in \R, \qquad t \ge0.
\]
Under Assumption (A) the measure $\mu_t$ possesses a smooth Lebesgue density, which solves the partial differential equation
\bee
\frac{\partial}{\partial t} \mu_t = \frac 12 \frac{\partial^2}{\partial x^2} \mu_t + \frac 12 
\frac{\partial}{\partial x} ( (\vp' \star \mu_t) \mu_t ), \qquad \mu_0 (dx)=\P(\overline{X}_{0}\in dx).
\eee 
The stochastic differential equation \eqref{meanfield} admits an invariant density $\pi$, which is described by an integral equation of convolution type:
\bee \label{pidef}
\pi(x) = Z^{-1}_{\pi} \exp\left(-(\vp \star \pi)(x) \right) \qquad \text{with} \quad 
Z_{\pi}=\int_{\R} \exp\left(-(\vp \star \pi)(x) \right) dx.   
\eee
In this article we will consider semiparametric estimation of the interaction function $\vp'$ and the identity \eqref{pidef} will be key for our approach. 

Next, we will demonstrate a propagation of chaos result for the particle system \eqref{Yprocess}. We recall that the \textit{Wasserstein $p$-distance} between two probability measures $\mu_1, \mu_2$ on $\R$ is defined by
\[
W_p(\mu_1, \mu_2):= \Big( \inf_{X_1\sim \mu_1, X_2\sim \mu_2} \E[|X_1-X_2|^p] \Big)^{1/p},
\]
where the infimum is taken over all couplings $(X_1,X_2)$ such that $X_i$ has the law $\mu_i$, $i=1,2$. The following theorem has been shown in \cite{cattiaux2008probabilistic,malrieu2003convergence}.

\begin{theo}[Theorems 5.1 and 6.2 in \cite{malrieu2003convergence}] \label{th1}
Let $\overline{X}^i$, $1\le i \le N$, be i.i.d.\ copies of the process $\overline{X}$ defined at 
\eqref{meanfield} so that every $\overline{X}^{i}$ is driven by the same Brownian motion as the $i$th particle of the system \eqref{eq:sde-part} and they are equal at time $0$. Denote by 
$$\Pi_{N,T}= N^{-1} \sum_{i=1}^N \de_{Y_{T}^{i,N}}$$ 
the empirical distribution of the projected particle system $Y_{T}^{i,N}$, $1 \le i \le N$, and by $\Pi$ the law associated to the invariant density $\pi$.  Under Assumption (A) there exist constants $C_1,C_2>0$ (independent of $N,T$) such that 
\bee
\sup_{t\geq 0} \E[|Y_{t}^{i,N} - \overline{X}^i_t|^2] \leq C_1 N^{-1}
\eee
and
\bee \label{wasser}
\E[W^2_1(\Pi_{N,T}, \Pi)] \leq C_1 N^{-1} + C_2 \exp(-\la T)=: N_T^{-1}
\eee
where the constant $\la>0$ has been introduced in \eqref{convex}.
\end{theo}

\noindent
The estimate \eqref{wasser} states that the invariant distribution $\Pi$ of the mean field equation 
\eqref{meanfield} is well approximated by the empirical measure $\Pi_{N,T}$ and gives the error bound associated with this approximation. In the next section we will use this result in our estimation procedure.

\section{Statistical problem and the estimation method} \label{sec3}
\setcounter{equation}{0}
\renewcommand{\theequation}{\thesection.\arabic{equation}}

We assume that the data
\[
Y_{T}^{1,N}, \ldots, Y_{T}^{N,N}
\]
is observed and $N,T \to \infty$ (and, as a consequence, $N_T \to \infty$), and our goal is to estimate 
the interaction function $\vp'$ introduced in \eqref{eq:sde-part} in a semiparametric setting.  We remark that the sampling scheme is rather unusual as we observe the particle system only at the terminal time $T$. According to the identity  \eqref{pidef} and the statement \eqref{wasser}, the considered data suffices to identify the interaction function $\vp'$ via the mean field limit. 
\par
Due to the complexity of the integral equation  \eqref{pidef}, which we will heavily rely on in our estimation procedure, we can not treat fully general interaction functions $\vp'$. Instead we consider a semiparametric model of the form
\bee \label{sempar}
\vp (y) = \sum_{0 < j \le J} a_j \vp_j (y) + \be (y), \qquad y \in \R,
\eee
where
\bee
\vp_j (y) = y^{2j}, \ 0 < j \le J_1, \qquad \vp_j (y) = \cos (\te_j y),  \ J_1 < j  \le J, \qquad y \in \R,
\eee
for some known distinct frequencies 
$\te_{J_1+1} >0$, \dots, $\te_{J} >0$
and known positive integers $J_1 \le J$. 
The parameters $a_1>0$, $a_2 \ge 0$, \dots, $a_{J_1-1} \ge 0$, $a_{J_1}>0$, $a_{J_1+1} \in \R$, \dots, $a_J \in \R$
and the function $\be$ are unknown.
The nonparametric component $\be\in C^{2}(\R)$ is even, bounded and such that $\be'$ is bounded, $\| \be' \|_{L^1(\R)} := \int_{\R} |\be'(y)| dy <\infty$ (later we will also
assume that $\| \be'' \|_{L^2(\R)} := (\int_{\R} |\be''(y)|^2 dy)^{1/2} < \infty$).
The strict convexity condition 
\eqref{convex} is induced by the assumption
\bee
2a_1 - \sum_{J_1 < j \le J}  \theta_j^2 |a_j| + \be''(y) \geq \la>0, \qquad \forall y \in \R.
\eee
The presence of the polynomial term in \eqref{sempar}  gives upper and lower bounds for $\pi$, which are required in the proofs (and are difficult to obtain in the general setting). The presence of the trigonometric terms is for modelling purpose only and does not influence the estimation theory.

Our approach will be based upon the integral equation  \eqref{pidef}. We will first provide the representation of the invariant density $\pi$ in the setting \eqref{sempar}. In the following,
for any function $f\in L^1(\R)$, we denote by $\mathcal{F}(f)$ the Fourier transform of $f$. 

\begin{lem} \label{lem1} 
For $\vp$ given in \eqref{sempar}, we have
\bee
(\vp \star \pi) (y) = \al_0 + \sum_{0 < j \le J} \al_j \vp_j (y) + (\be \star \pi)(y), \qquad y \in \R,
\eee
with $\al_0 = \sum_{0< k \le J_1} m_{2k} a_k$ and
\bee\label{def:al}
\al_j = \sum_{j \le k \le J_1} \binom{2k}{2j} m_{2(k-j)} a_k, \ 0 <  j \le J_1, \qquad \al_j = a_j \mathcal{F}(\pi) (\te_j), \ J_1 < j \le J,
\eee
where $m_k=\int_{\R} y^k \pi(y) dy$ denotes the $k$th moment of the invariant measure $\Pi$.
\end{lem}

\begin{proof}
We consider the sum of $a_k (\vp_k \star \pi) (x)$ over $0<k\le J_1$, where
\bee 
(\vp_k \star \pi)(y) = \int (y-x)^{2k} \pi(x) dx =  \sum_{j=0}^k
\binom{2k}{2j} m_{2(k-j)} y^{2j}  
\eee
since $\pi$ is symmetric. Then interchanging the order of summation, we get the formula for the coefficients $\al_j$, $0 \le j \le J_1$. Similarly,  for $J_1<j \le J$, the coefficients $\al_j$ are obtained through the identity
\bee
\int \cos(\te_j (y-x)) \pi(x) dx = \cos(\te_j y) \int \cos(\te_j x) \pi(x) dx,
\eee
where we again have used the symmetry of $\pi$. This completes the proof of Lemma \ref{lem1}. 
\end{proof}

\noindent
We now proceed with the introduction of the estimation procedure, which consists of four steps:

\begin{itemize}
\item[(i)] Estimate the derivative of the log-density
\bee \label{lfunction}
l(y):= (\log \pi)'(y)= \frac{\pi'(y)}{\pi(y)}, \qquad y \in \R,
\eee
via a kernel-type estimator $l_{N,T}$ based on the observed data $Y_{T}^{1,N}, \ldots, Y_{T}^{N,N}$.
\item[(ii)] Estimate the parameter  $\boldsymbol{\al}:=(\al_1,\dots,\al_J)^\top \in \R^J$ using the minimum contrast method based on 
\bee 
l(y,\boldsymbol{\al}) = - \sum_{j=1}^J \al_j \vp'_j(y),
\eee 
which approximates $l(y) = l(y,\boldsymbol{\al}) - (\be'\star \pi)(y)$ for large values of $y \in \R$.
\item[(iii)]  Use the results of step (ii) to construct an estimator $\Psi_{N,T}$ of 
\bee
\Psi(y):= -(\be' \star \pi)(y), \qquad y \in \R.
\eee
\item[(iv)] Finally, apply the deconvolution 
\bee
\mathcal{F} (\be') (z)= - \frac{\mathcal{F} (\Psi )(z)}{\mathcal{F}(\pi)(z)} = - \frac{\mathcal{F} (\Psi )(z) \overline{\mathcal{F}(\pi)(z)}}{|\mathcal{F}(\pi)(z)|^2} , \qquad z \in \R,
\eee
and Fourier inversion to obtain an estimator $\be'_{N,T}$ of $\be'$.
\end{itemize}

\noindent
\textbf{Estimation of the function $l$.}
Following the latter we first introduce kernel estimators of $\pi$ and $\pi'$. For any function $f: \R\to \R$ and $u>0$ we use the standard notation
$f_u(x):= u^{-1} f(u^{-1} x)$. 
Let $K$ be a smooth kernel of order $m \ge 2$, that is  
\[
\int_{\R} K(x)dx=1, \qquad \int_{\R} x^{j}K(x)dx=0 \quad j=1,\ldots, m-1, \qquad  \int_{\R} x^{m}K(x)dx \not =0.
\]
Let $h_i = h^i_{N,T}$, $i=0,1$, be two bandwidth parameters vanishing as $N,T \to \infty$. We define 
\bee \label{kernelest}
\pi_{N,T} (y) := \frac{1}{N} \sum_{i=1}^{N} K_{h_0} \Big(y - Y_T^{i,N}\Big),
\quad
\pi'_{N,T} (y) := \frac{1}{N h_1} \sum_{i=1}^{N} K'_{h_1} \Big(y - Y_T^{i,N}\Big), \quad y \in \R.
\eee  
Next, we introduce a threshold $\de = \de_{N,T} \to 0$ as $N,T\to\infty$ and set
\bee \label{lnt}
l_{N,T}(y):= \frac{\pi'_{N,T} (y)}{\pi_{N,T} (y)} \1_{\{\pi_{N,T} (y)>\de \}}, \qquad y \in \R,
\eee
which is an estimator of the function $l$ given at \eqref{lfunction}. \\ \\
\textbf{Estimation of the parametric part.} Recall the identity  $l(y) = l(y,\boldsymbol{\al})-(\be'\star \pi)(y)$. Using $(\be'\star \pi)(y) \to0$ as $|y|\to \infty$ since $\be' \in L^1(\R)$, we will construct the minimum contrast estimator for $\boldsymbol{\al}$. More specifically, we introduce an integrable weight function $w$ with support $[\epsilon,1]$ ($\epsilon\in (0,1)$) and a parameter $U = U_{N,T}\to \infty$ as $N,T\to\infty$. For $\boldsymbol{\al} \in \R^J$, we define
\bee \label{alrho}
\boldsymbol{\al}_{N,T} := \arg \min_{\boldsymbol{\al}\in \R^J} \int (l_{N,T} (y) - l(y,\boldsymbol{\al}) )^2  w_U (y) dy.
\eee  
	We can use the relations \eqref{def:al} to estimate the coefficients $\boldsymbol{a} = (a_1,\dots,a_J)^\top$, based on the empirical moments $m_{2j; N,T}$, $1 < j < J_1$, 
	and the empirical Fourier moments $\f (\Pi_{N,T})(\te_j)$, $J_1 < j \le J$, of the particle system:
	\bee
	m_{k; N,T} := \frac{1}{N} \sum_{i=1}^N (Y^{i,N}_T)^k, \ k \in \N, \qquad 
	\mathcal{F}(\Pi_{N,T})(z):=\frac 1N \sum_{i=1}^N \exp (\mathrm{i} z Y_T^{i,N} ), \ z \in \R.
	\eee
	By solving the corresponding linear systems we so construct estimates $\boldsymbol{a}_{N,T}$ for the coefficients $\boldsymbol{a}$.
\\ \\
\textbf{Estimation of the nonparametric part.} Given the estimator $\boldsymbol{\al}_{N,T}$ introduced
in the previous step, we define
\bee
\Psi_{N,T}(y) = ( l_{N,T} (y) - l (y, \boldsymbol{\al}_{N,T}) ) \1_{\{|y|\leq \epsilon U\}}, \qquad y \in \R,
\eee
which provides an estimator of the function $\Psi = -\be' \star \pi$. In the next step we choose another threshold $\om = \om_{N,T} \to 0$ and introduce 
\bee
\mathcal{F}(\be'_{N,T}) (z):= -\frac{\mathcal{F}(\Psi_{N,T}) (z) \overline{\mathcal{F}(\Pi_{N,T})(z)} }{|\mathcal{F}(\Pi_{N,T}) (z)|^2} \1_{\{|\mathcal{F}(\Pi_{N,T})(z)|>\om \}}, \qquad z \in \R.
\eee
Finally, we use the Fourier inversion to estimate the function $\be'$:
\bee\label{def:h}
\be'_{N,T}(y):= \frac{1}{2\pi} \int \exp(-\mathrm{i} zy) \mathcal{F}(\be'_{N,T})(z) dz, \qquad y \in \R.
\eee
In the following we will derive asymptotic properties of all estimators introduced in this section.

\section{The asymptotic theory} \label{sec4}
\setcounter{equation}{0}
\renewcommand{\theequation}{\thesection.\arabic{equation}}

We start our asymptotic analysis with the estimator $l_{N,T}$. In the following the bandwidth and threshold parameters are chosen as
\bee \label{bandwidth}
h_0 = N_T^{-\frac{1}{2(m+1)}}, \qquad h_1 = N_T^{-\frac{1}{2(m+2)}}, \qquad \de = \de_0 \exp (- \bar \al_1 U^{2J_1} ),
\eee 
where $N_T$ is the rate introduced in \eqref{wasser}, $m$ is the order of the kernel $K$ and
$\de_0 := (2 Z_\pi )^{-1} \exp ( - \al_0 - \sum_{J_1 < j \le J} |\al_j| - \| \be \|_\infty)$,  
$\bar \al_1 :=\sum_{0 < j \le J_1} \al_j$.
Here and in what follows, $\| f \|_\infty := \sup_{y \in \R} |f(y)|$ for $f : \R \to \R$. Furthermore, we write $a_n \lesssim b_n$ when there exists a constant $C>0$, independent of $n$, such that $a_n\leq Cb_n$. Our first result is the following proposition.

\begin{prop} \label{prop1}
Let $\de$, $h_i$, $i=0,1$, be defined as in \eqref{bandwidth} and $U \ge 1$. Then
\bee
\sup_{|y|\leq U}  \E \left[ |l_{N,T}(y) - l(y)|^2 \right]^{\frac 1 2} \lesssim \exp ( \bar \al_1 U^{2J_1} ) \left(
N_T^{-\frac{m}{2(m+2)}} + U^{2J_1-1} N_T^{-\frac{m}{2(m+1)}} \right).
\eee
\end{prop}
\begin{proof} See Section  \ref{sec6}.
\end{proof}

\noindent
We observe that the upper bound in Proposition \ref{prop1} grows exponentially in $U$, which will strongly affect convergence rates for all parameters of the model.
We will now find the explicit expression for the estimator of $\boldsymbol{\al}$. 
For this purpose, in \eqref{lnt}  we write $l (y,\boldsymbol{\al}) = \boldsymbol{l}(y/U)^\top \cdot \boldsymbol{\al}^{U}$, where
\bee
\boldsymbol{l} (y) = -(\vp'_1 (y), \dots,\vp'_{J_1} (y), \vp'_{J_1+1}(Uy), \dots, \vp'_J(Uy))^\top, \qquad y \in \R,
\eee
and
\bee
\boldsymbol{\al}^{U} = (\al_1 U, \al_2U^3, \dots, \al_{J_1} U^{2J_1-1}, \al_{J_1+1}, \al_{J_1+2}, \dots, \al_J )^\top.
\eee
Then the unknown $\boldsymbol{\al}^{U}$ and analogously scaled estimator $\boldsymbol{\al}_{N,T}^{U}$ satisfy the identities
\bee
Q \boldsymbol{\al}_{N,T}^{U} = \int l_{N,T} (y) \boldsymbol{l} (y/U) w_U(y) dy, \qquad Q \boldsymbol{\al}^{U} = \int (l(y)+(\be' \star \pi)(y)) \boldsymbol{l} (y/U) w_U(y) dy,
\eee
where 
\bee \label{Qdef}
Q:= \int \boldsymbol{l} (y) \boldsymbol{l} (y)^\top w (y) d y \in \R^{J\times J}.
\eee
Notice that the components of the function $\boldsymbol{l}$ are linearly independent on any 
interval $[s,t]$, $s<t$, since
$U\te_{J_1+1},\dots, U\te_{J}$ are all distinct. 
In this case the matrix $Q$ is positive definite and hence invertible. 

In the next proposition we derive convergence rates for the estimates $\boldsymbol{\al}_{N,T}^{U}$
and $\Psi_{N,T}$. By $\| \cdot \|_2$ we denote the Euclidean norm.

\begin{prop}\label{prop2}
	Let $\de$, $h_i$, $i=0,1$, be defined as in \eqref{bandwidth} and $U \ge 1$.  Then
	\begin{align*}
	&\E \left[ \| \boldsymbol{\al}_{N,T}^{U}- \boldsymbol{\al}^{U} \|^2_2\right]^{\frac 1 2} \lesssim \exp (\bar \alpha_1 U^{2J_1}) 
	\left(N_T^{-\frac{m}{2(m+2)}} + U^{2J_1-1} N_T^{-\frac{m}{2(m+1)}} \right) \\ 
	&\qquad\qquad+ \frac{\exp (- \al_{J_1} (\epsilon U/2)^{2J_1})}{(\epsilon U/2)^{2J_1}} + U^{-1} \int_{|y|>\epsilon U/2} |\be' (y)| d y,\\
	& \E \left[ \int_{\R} | \Psi_{N,T} (y) - \Psi (y) |^2 d y \right]^{\frac 1 2}\lesssim \exp (\bar \al_1 U^{2J_1}) U^{\frac 1 2}    \left(N_T^{- \frac{m}{2(m+2)}} + U^{2J_1-1} N_T^{-\frac{m}{2(m+1)}} \right)\\ 
	&\qquad\qquad + \frac{\exp(-\al_{J_1} (\epsilon U/2)^{2J_1})}{(\epsilon U/2)^{2J_1-1}} + U^{-\frac{1}{2}} \int_{|y|>\epsilon U/2} |\be' (y)| d y+\left( \int_{|y|>\epsilon U/2} |\be'(y)|^2 dy \right)^{\frac{1}{2}}.
	\end{align*}
\end{prop}
\begin{proof} See Section  \ref{sec6}.
\end{proof}

\noindent
We remark that the convergence rates for the parametric part of the model are logarithmic, which is rather unusual for parametric estimation problems.  
Indeed, the logarithmic rate is obtained under a proper choice of $U$ when there exists $q>0$ such that $0<\liminf_{y \to \infty} F_{q,\be'} (y) \le \limsup_{y \to \infty} F_{q,\be'} (y) <\infty$, where $F_{q,\be'}(y)=y^q \int_y^\infty |\be'(u)| du$, $y>0$.
We believe that the reason for such a slow convergence rate is a convolution type structure of the invariant density $\pi$.

\begin{rem} \label{rem1} \rm
The statement of Proposition \ref{prop2} can be transferred from $\boldsymbol{\al}$ to the original parameter $\boldsymbol{a}$ via the identities \eqref{def:al}. 
This follows from the estimate $\E [ |m_{2j; N,T} - m_{2j}|^{2k}]  \lesssim N_T^{-1}$ when $X_0^1$ has sufficiently high moments, see e.g.\ \cite{cattiaux2008probabilistic,malrieu2003convergence}. \qed
\end{rem}

\noindent
It is evident from Proposition \ref{prop2}  that the rate of convergence for $\Psi_{N,T} $ crucially depends on the tail behaviour of the function $\be'$.  
The next corollary, which is an immediate consequence of the previous result, gives the precise bound. 
 
\begin{cor}\label{cor1}
	 In the setting of Proposition \ref{prop2}  assume that there exists $p>0$ such that $\liminf_{y \to \infty} \Phi_{p,\be'} (y)>0$, where 
	 \bee\label{def:Phip}
	 \Phi_{p,\be'} (y):= y^{p-(1/2)} \int_y^\infty |\be'(u)| du + y^p \Big( \int_y^\infty |\be'(u)|^2 du \Big)^{1/2}, \qquad y>0.
	 \eee
	Choose $U = (c \log N_T)^{1/(2J_1)}$ for some $0< c < m / (2 (m+2) \bar \al_1)$. Then 
	\bee
	\E \left[ \int_{\R} | \Psi_{N,T} (y) - \Psi (y) |^2 d y \right] \lesssim (\log N_T)^{-p/J_1} \Phi^2_{p,\be'} ( (c \log N_T )^{1/(2J_1)} \epsilon/2 ).
	\eee
	Moreover, if $\limsup_{y \to \infty}\Phi_{p,\be'} (y)<\infty$, then
	\bee \label{psiconv}
	\E \left[ \int_{\R} |\Psi_{N,T} (y)-\Psi (y)|^2 d y \right]\lesssim (\log N_T)^{-p/J_1}.
	\eee
\end{cor}

\noindent
We see that the convergence rate in  \eqref{psiconv} depends on the nonparametric part of the model through the parameter $p$ and on the highest degree polynomial in the parametric part through $J_1$ (in contrast, the trigonometric part of the model does not influence the convergence rate). 
This phenomenon will translate to the estimation of $\be'$.

The following proposition will play a crucial role for the analysis of the estimator $\be'_{N,T}$.

\begin{prop}\label{prop3}
	In the setting of Proposition \ref{prop2} assume $\be'' \in L^2(\R)$. Then
	\begin{align}
	\E\left[\int_{\R} | \be'_{N,T}(y)-\beta' (y) |^2 d y \right]  &\lesssim \omega^{-2}  \Big(  \E\left[\int_{\R} | \Psi_{N,T}(y)-\Psi(y) |^2 d y \right] + N_T^{-1} \Big)\nonumber\\ 
	&\qquad + N_T^{-1} \int_{|\mathcal{F}(\pi) (z)| > 2\om} |\mathcal{F}(\be') (z)|^2 |\mathcal{F}(\pi)(z)|^{-2} z^2 d z\\ 
	&\qquad+ \int_{|\mathcal{F}(\pi) (z)| \le 2 \om} |\mathcal{F}(\be') (z)|^2 d z.\label{int:betaMISE}
	\end{align}
\end{prop}
\begin{proof} See Section  \ref{sec6}.
\end{proof}

\noindent
We observe that, up to the presence of the factor $\om^{-2}\to \infty$ which can be chosen arbitrarily, the convergence rate for the nonparametric part 
$\be'$ is transferred from Proposition \ref{prop2}. Indeed, the next result shows that the second and the third terms in \eqref{int:betaMISE} are negligible
under appropriate assumption on $\be'$.

\begin{prop}
	\label{prop:rate-entire}
	In the setting of Proposition \ref{prop3} assume that $\be'$ is an entire function of the first order and type less than $\vt>0$, i.e.
	\bee
	|\be'(z)| \le A \exp (  \vt |z| ), \qquad z \in \CC, 
	\eee
	with $A >0$ and $\vt \le \lambda^{1/2}$. 
	Let $\liminf_{y \to \infty} \Phi_{p,\be'} (y)>0$ for $p>0$.
	Choose $U = (C \log N_T)^{1/(2J_1)}$ for some $0< C < m/(2(m+2) \bar \alpha_1)$. Then 
	\begin{equation}\label{MISEbeta}
	\E\left[ \int_{\R} |\be'_{N,T} (y) -\be' (y)|^2 d y\right] \lesssim \om^{-2} (\log N_T)^{-p/J_1} \Phi_{p,\be'}^2 ( ( C \log N_T )^{1/(2J_1)} \epsilon/2 ).
	\end{equation}
	Moreover, if $\limsup_{y \to \infty} \Phi_{p,\be'} (y)>0$, then
	\bee
	\E\left[\int_{\R} | \be'_{N,T}(y)-\be' (y) |^2 d y \right]  \lesssim \om^{-2} (\log N_T)^{-p/J_1}.
	\eee
\end{prop}
\begin{proof} See Section  \ref{sec6}.
\end{proof}

\begin{ex}
	As an example of functions satisfying the conditions of Proposition \ref{prop:rate-entire} we may consider 
	\bee
	\be_1(y) = \frac{1-\cos(b y)}{y^2}, \ b>0, \qquad \be_2(y) = \frac{\sin^{2k}(y)}{y^{2k}}, \ k \in \N.
	\eee
	One can easily check that both function fulfil the necessary assumptions with $p=3/2$ 
	in case of $\be_1$ and 
	$p=2k-1/2$ 
	in case of $\be_2$.
	\qed
\end{ex}

\section{Lower bound for the nonparametric component} \label{sec5}
\setcounter{equation}{0}
\renewcommand{\theequation}{\thesection.\arabic{equation}}

In this section we will derive the lower bound for the estimation of the nonparametric component $\be'$. We consider the simple model with i.i.d.\ observations $Y_1, \dots, Y_N$ having the following density:
\bee \label{iidmodel}
\pi_{\be}(y)= Z_{\pi_{\be}}^{-1} \exp \Big(- \Big( \Big( \sum_{j=1}^J a_j \vp_j + \be \Big) \star \pi_{\be} \Big) (y) \Big), \qquad y \in \R,
\eee
where
\bee 
Z_{\pi_\be} =  \int_{\R} \exp \Big(- \Big( \Big( \sum_{j=1}^J a_j \vp_j + \be \Big) \star \pi_{\be} \Big)(y) \Big) dy
\eee
and $\vp_j (y) = y^{2j}$, $y \in \R$, $1 \le j \le J$, for given $J \ge 1$. We assume that the constants $a_1 >0$, $a_2 \ge 0$, \dots, $a_{J-1} \ge 0$, $a_J>0$ are known, the function $\be(y)$ is even and such that $\vp''(y) \ge \la$, $y \in \R$, with known $\la >0$.
We remark that the nonparametric estimation problem is similar in spirit to the classical deconvolution problem, but we can not apply the same techniques to derive the lower bound. 
\par
We consider the following class of functions 
\begin{align*}
\f_{p,C, C_0, \boldsymbol{a},\la} := \Big\{\Big.f \in C_b^{2}(\R): ~&\|f\|_{\infty} \le C_0,~\|f' \|_\infty \le C_1,~\inf_{y \in {\R}} f''(y)\geq  -C_2\\
\qquad\qquad\qquad &\limsup_{y \to \infty} y^{2p} \int_y^\infty |f'(u)|^2 du \le C \Big\}\Big.,
\end{align*}
where $C_b^2(\R)$ denotes the space of twice continuously differentiable functions $f : \R \to \R$ such that
$f,f',f''$ are bounded. Moreover, $C,C_0>0$ and
\bee
C_1 := \la^{1/2} \Big(1 - \sum_{1 < j \le J} 2j c_j a_j / \la^j \Big)>0, \qquad C_2 := 2a_1-\la>0,
\eee 
where
$c^2_j := 2 (2(2j-1))! / (2j-1)!$, $1<j\le J$. We remark that the condition $C_1>0$ holds whenever $\la>0$ is sufficiently large. 
We also note that assumption (A) is automatically satisfied when $\be \in \f_{p,C, C_0, \boldsymbol{a},\la}$. The main result of this section is the minimax bound over the functional class $\f_{p,C, C_0, \boldsymbol{a},\la}$.

\begin{theo} \label{minimax}
	Denote the law of $Y_1$ by $\Pi_{\be}$, and consider the functional class $\f_{p,C, C_0, \boldsymbol{a},\la}$ for 
	$p>1/2$. Then there exists a constant $c_0>0$ (depending on $p$, $C$, $C_0$, $\boldsymbol{a}$, $\la$) such that
	\bee \label{minimaxrate}
	\inf_{\be'_N} \sup_{\be \in \f_{p,C, C_0, \boldsymbol{a},\la}} \Pi_{\be}^{\otimes N}\Big( \|\be'_N - \be'\|_{L^2(\R)}^2 >c_0 (\log N)^{-p/J} \Big)>0.
	\eee
\end{theo}
\begin{proof} See Section  \ref{sec6}.
\end{proof}

\noindent
We remark that our estimator $\be'_{N,T}$ matches the lower bound up to the factor $\om^{-2}$, which can be chosen to diverge to $\infty$ at an
arbitrary slow rate. 

\begin{rem} \label{rem2} \rm
The result of Theorem \ref{minimax} can be compared to a classical deconvolution problem. Consider a model
\[
Y_i=X_i +\ep_i, \qquad i=1,\ldots, N,
\]
where $(X_i)_{i\ge 1}$ and $(\ep_i)_{i\ge 1}$ are mutually independent i.i.d.\ sequences. Assume that $X_1$ (resp. $\ep_1$) has a Lebesgue density $f$
(resp. $g$), and we are in the super smooth setting, that is,
\[
\f(g)(z) \sim \exp(-\text{const} \cdot |z|^{2J}), \quad |z|\to \infty.
\] 
When the density $f$ satisfies the condition $\int_{\R} |\f(f)(z)|^2 |z|^{2p} dz \leq C$, $p>1/2$, it is well known that the minimax rate for the estimation of the density $f$ becomes $(\log N)^{-p/J}$ (see e.g. Theorem 2.14(b) in \cite{M09}). While in the classical deconvolution problem the assumptions are imposed in the Fourier domain, we have comparable assumptions on the functions themselves. 
Notice that due to the structure of the model $\pi$ plays the role of the noise density and
the condition on  $\f(g)$ can be compared to the decay of $\pi_{\be}$, which is determined by the leading polynomial of degree $J$. On the other hand, the integral condition on $\f(f)$ is related to the corresponding tail condition on $\be'$. 
\end{rem}

\section{Conclusions and outlook} \label{secOut}
In this work we study the problem of estimating the drift of a MV-SDE based on observations of the corresponding particle system. We propose a kernel-type estimator and provide theoretical analysis of its convergence. 
In particular, for the nonparametric part of the model, we derive minimax convergence rates and show rate-optimality of our proposed estimator.
As a promising future research direction, one can consider the case of continuous time observations and high-dimensional MV-SDEs with general form of the drift function. In another direction the problem of estimating diffusion coefficient remains completely open.

\section{Proofs} \label{sec6}
\setcounter{equation}{0}
\renewcommand{\theequation}{\thesection.\arabic{equation}}

\subsection{Preliminary results}

\begin{lem}\label{lem2}
	Set $\bar \vp_1 (y): = \sum_{0 < j \le J_1} \al_j \vp_j(y)$, $y\in\R$. Then $\pi \in C^\infty (\R)$ and for every $n$, 
	\bee
	|\pi^{(n)} (y)| \lesssim (1+|\bar \vp'_1 (y)|)^n \exp (- \bar \vp'_1 (y)), \qquad y\in\R.
	\eee
\end{lem}

\begin{proof}
	We decompose $\bar \vp = \vp \star \pi$ as $\bar \vp = \al_0 + \bar \vp_1 + \bar \vp_2 + \be \star \pi$. Here $\bar \vp_2: = \sum_{J_1 < j \le J} \al_j \vp_j$ is bounded, and $\| \be \star \pi \|_\infty \le \| \be \|_\infty \| \pi \|_{L^1(\R)} < \infty$. Hence, we obtain
	\bee
	\pi (y) = Z_\pi^{-1} \exp (- \bar \vp (y)) \lesssim \exp (- \bar \vp_1 (y)), \qquad y \in \R.
	\eee
	We now consider its derivative. Since $\be$ has a bounded derivative,  so does $\be \star \pi$. We obtain
	$\pi' = - \bar \vp' \pi $, where 
	$\bar \vp' = \bar \vp'_1 + \bar \vp'_2 + \be'\star \pi$ satisfies
	\bee
	| \bar \vp' (y) | \lesssim 1+|\bar \vp'_1(y)|, \qquad y \in \R.
	\eee
	That is, statement of the proposition holds for $n=1$. 	
	For $n \ge 1$ it follows by induction using
	\bee
	\pi^{(n+1)} = (\pi')^{(n)}= - \sum_{k=0}^{n} \binom{n}{k} \bar \vp^{(k+1)} \pi^{(n-k)} 
	\eee
	with $\bar \vp^{(k+1)} = \bar \vp_1^{(k+1)}+\bar \vp_2^{(k+1)}+(\be'\star \pi)^{(k)}$, where $(\be'\star \pi)^{(k)} = \be' \star \pi^{(k)}$ is bounded when $\| \pi^{(k)} \|_\infty < \infty$, $\| \be' \|_{L^1(\R)} < \infty$. 
\end{proof}

\begin{lem}\label{mom}
	Moments of the density $\pi$ in \eqref{pidef} satisfy 
	\bee
	m_{2k} \le (2k-1)!! / \la^k, \qquad k \in \N.
	\eee
\end{lem}

\begin{proof}
	Set
	\bee
	I_k := \int_{\R} y^{2k-1} (\vp' \star \pi) (y) \pi (y) dy. 
	\eee
	Since $\vp(y)$ is even, we have $\vp'(y)= -\vp'(-y)$, implying
	\bee
	2I_k = \int \int (y^{2k-1}-x^{2k-1}) \vp' (y-x) \pi (x) \pi (y) dx dy.
	\eee
	By the mean value theorem we conclude that
	\bee 
	\vp'(y-x) = \vp'(y-x) - \vp'(0) = (y-x) \vp''(z) 
	\eee 
	for some $z \in (x,y)$, and by the convexity assumption,
	\bee
	2I_k \ge \la \int \int (y^{2k-1}-x^{2k-1})(y-x) \pi (x) \pi (y) dx dy = 2 \lambda m_{2k}.
	\eee
	On the other hand,
	\bee 
	I_k = - \int y^{2k-1} d \pi (y) = (2k-1) \int y^{2(k-1)} \pi (y) d y = (2k-1) m_{2(k-1)}.
	\eee
	We conclude that $m_{2k} \le m_{2(k-1)} (2k-1)/\la$. Induction provides the desired result. 
\end{proof}

\subsection{Proof of Proposition \ref{prop1}}
	Using
	\bee
	|l_{N,T} (y)-l(y)| \le \Big| \frac{\pi'_{N,T}(y)}{\pi_{N,T}(y)} - \frac{\pi'(y)}{\pi_{N,T}(y)} \Big| + \Big| \frac{\pi'(y)}{\pi_{N,T}(y)} - \frac{\pi'(y)}{\pi(y)} \Big|,
	\eee
	on $\pi_{N,T} (y) > \de$, we get
	\begin{align*}
	|l_{N,T} (y)-l(y)| &\le  \de^{-1} |\pi'_{N,T}(y)-\pi'(y)| + \de^{-1} |l(y)| |\pi(y)-\pi_{N,T}(y)| + |l(y)| \1_{\{ \pi_{N,T} (y) \le \de\}}.
	\end{align*}
	Here using $|l(y)| \le 2 \sum_{0 < j \le J_1} j \al_j |y|^{2j-1} + \sum_{J_1 < j \le J} \te_j | \al_j | + \|\be'\|_\infty$
	we get
	\bee
	|l(y)| \lesssim U^{2J_1-1}
	\eee
	for all $|y| \le U$. Since $|(\vp \star \pi)(y)| \le \sum_{0 \le j \le J_1} \al_j y^{2j} + \sum_{J_1 < j \le J} |\al_j| + \| \be \|_\infty$, the chosen $\de$ satisfies
	\bee
	2 \de \le Z_\pi^{-1}  \exp ( - (\vp \star \pi) (y) )  = \pi (y)
	\eee
	for all $|y| \le U$. Hence, it follows that for all $|y| \le U$, 
	\begin{align*}
	\P ( \pi_{N,T} (y) \le \de ) 
	&= \P ( \pi (y) -\pi_{N,T} (y) \ge \pi (y)- \de )\\ 
	&\le \P ( \| \pi -\pi_{N,T} \|_\infty \ge \de ) \le \de^{-2} \E [\| \pi - \pi_{N,T} \|_\infty^2].
	\end{align*}
	Now, with $\pi_{N,T}(y) = (K_{h_0} \star \Pi_{N,T}) (y)$, we have 
	\begin{align*}
	&| (K_{h_0} \star (\Pi_{N,T} - \Pi)) (y) | 
	\le \operatorname{Lip}(K_{h_0} (y - \cdot )) W_1 (\Pi_{N,T}, \Pi)  = h_0^{-1} \operatorname{Lip}(K)  W_1 (\Pi_{N,T}, \Pi),
	\end{align*}
	where
	$\E [W_1^2 (\Pi_{N,T},\Pi)] \le N_T^{-1}$ due to  \eqref{wasser}.
	Furthermore, we have
	\bee
	(K_{h_0} \star \Pi) (y) - \pi (y) = \int K (x) (\pi (y+xh_0) - \pi (y)) dx,
	\eee
	where $\pi \in C^\infty (\R)$ satisfies
	$|\pi^{(n)} (y) | \lesssim (1+|y|^{2J_1-1}) \exp (-\al_{J_1} y^{2J_1})$, $y\in \R$, $n \in \N$, by Lemma~\ref{lem2}.
	Using a Taylor expansion of $\pi$ and that the kernel $K$ is of order $m$, we obtain
	\bee
	(K_{h_0} \star \Pi) (y) - \pi (y) = \int K(x) R_{m-1} (y+xh_0) dx,
	\eee
	where
	\bee
	|R_{m-1}(y+xh_0)| \lesssim |x|^m h_0^m
	\eee 
	and so
	\bee
	|(K_{h_0} \star \Pi)(y) - \pi (y)| \lesssim h_0^m
	\eee
	uniformly in $y \in \R$.
	Similarly, 
	\bee
	| h_1^{-1} ( K'_{h_1} \star \Pi ) (y) - \pi' (y) | \lesssim h_1^m,
	\eee
	because
	\bee
	h_1^{-1} ( K'_{h_1} \star \Pi ) (y)  = ( (K_{h_1})' \star \Pi ) (y) 
	= \int K(x) \pi' (y+xh_1) dx.
	\eee
	We thus deduce that
	\bee
	\E [\| \pi_N - \pi \|_\infty^2] \lesssim N_T^{-1}h_0^{-2} + h_0^{2m}, \qquad \E[ \| \pi'_N - \pi' \|_\infty^2] \lesssim N_T^{-1} h_1^{-4} + h_1^{2m},
	\eee
	where our choice of $h_0, h_1$ yields the optimal rate in upper bounds:
	\bee
	\E[ \| \pi_N - \pi \|^2_\infty] \lesssim N_T^{-\frac{m}{m+1}}, \qquad  \E[ \| \pi'_N - \pi' \|^2_\infty] \lesssim N_T^{-\frac{m}{m+2}}.
	\eee

\subsection{Proof of Proposition \ref{prop2}}
	Recall the definition of the matrix $Q \in \R^{J\times J}$ introduced at \eqref{Qdef}. Since $Q$ is invertible, we deduce that
	\bee
	\E \left[ \| \boldsymbol{\al}^{U}_{N,T} -\boldsymbol{\al}^{U} \|_2^2 \right]^{\frac{1}{2}} 
	\le \int  \E [|l_{N,T}(y)-l(y)-(\beta' \star \pi)(y)|^2 ]^{\frac{1}{2}} \tilde w_U (y) d y
	\eee
	where   $\tilde w (y):= \| Q^{-1} \boldsymbol{l}(y) \|_2 |w(y)|$, $y\in\R$. Moreover, $\|\tilde w\|_{L^1(\R)} < \infty$
	and 
	$\|\tilde w\|_\infty < \infty$ are uniformly bounded in $U$. Hence,
	\bee
	 \E \left[ \| \boldsymbol{\al}^{U}_{N,T} -\boldsymbol{\al}^{U} \|_2^2 \right]^{\frac{1}{2}}
	\lesssim \sup_{|y| \le U}  \E[ |l_{N,T}(y) - l(y)|^2 ]^{\frac{1}{2}} + U^{-1} \int_{\epsilon U}^\infty |(\be' \star \pi) (y)| d y.
	\eee
	As for the last term, we have
	\begin{align*}
	&\int_{\epsilon U}^\infty | (\be' \star \pi) (y) | d y \le \int_{\epsilon U}^\infty \int |\be' (y-x)| \pi (x) d x d y\\
	&\qquad= \Big(\int_{-\infty}^{\epsilon U/2} + \int_{\epsilon U/2}^\infty \Big)
	\Big( \int_{U \epsilon-x}^\infty |\be' (y)| d y \Big) \pi (x) d x \le \int_{\epsilon U/2}^\infty |\be' (y)| dy + \| \be' \|_{L^1(\R)} \int_{\epsilon U/2}^\infty \pi (x) d x.
	\end{align*}
	Finally, note that 
	\bee
	\int_{u}^\infty \pi (x) dx \lesssim \int_u^\infty \pi_1 (x) dx,
	\eee
	where $\pi_1 (x) = \exp(- \al_{J_1} x^{2J_1})$ satisfies $\pi'_1 (x) = -2 J_1 \al_{J_1} x^{2J_1-1} \pi_1 (x)$ and so, 
	\bee
	\int_u^\infty \pi_1 (x) d x \le \frac{1}{u^{2J_1-1}} \int_u^\infty x^{2J_1-1} \pi_1 (x) d x = \frac{\pi_1 (u)}{2 J_1 \al_{J_1} u^{2J_1-1}} .
	\eee
	
	Now we consider the bound for 
	\bee
	\Psi_{N,T} (y) = ( l_{N,T} (y) - l (y,\boldsymbol{\al}_{N,T}) ) \1_{\{|y| \le \epsilon U\}}, \qquad y \in \R.
	\eee
	with
	\bee
	\Psi (y) = l(y) - l (y,\boldsymbol{\al}) = - (\be'\star \pi)(y), \quad y\in \R.
	\eee
	We have 
	\bee
	\int_{|y| \le \epsilon U} \E [|\Psi_{N,T}(y) -\Psi(y)|^2] d y \le \epsilon U\sup_{|y| \le \epsilon U} \E [|\Psi_{N,T}(y)-\Psi(y)|^2],
	\eee
	where
	\bee
	\sup_{|y| \le \epsilon U} \E[ |\Psi_{N,T} (y)-\Psi(y)|^2]
	\lesssim \sup_{|y|\le \epsilon U} \E [|l_{N,T}(y)-l(y)|^2] + \E[ \| \boldsymbol{\al}^{U}_{N,T} - \boldsymbol{\al}^{U} \|^2_2].
	\eee
	Finally, we deduce
	\begin{align*}
	\Big( \int_{|y|>\epsilon U} |\Psi (y)|^2 dy \Big)^{\frac{1}{2}} 
	&\le \int \Big( \int_{|y|>U\epsilon} |\be' (y-x)|^2 dy \Big)^{\frac{1}{2}} \pi (x) d x\\
	&\le \Big( \int_{|y|>\frac{\epsilon}{2}U} |\be' (y)|^2 dy \Big)^{\frac{1}{2}} + \| \be' \|_{L^2(\R)} \int_{|x|>\frac{\epsilon}{2}U}  \pi (x) d x,
	\end{align*}
	which completes the proof of Proposition \ref{prop2}.

\subsection{Proof of Proposition \ref{prop3}}
	Define a function $\be_{N,T}^*$ via the formula
	\bee
	\mathcal{F}(\be^\ast_{N,T}) (z) = \mathcal{F}(\beta') (z) \1_{\{|\mathcal{F}(\Pi_{N,T})(z)|>\om\}}.
	\eee
	Write $\mathcal{F}(\Pi) = \mathcal{F}(\pi)$. Use the identity
	\begin{align*}
	&\mathcal{F}(\be'_{N,T}-\be^\ast_{N,T}) (z)\\  
	&\qquad= (- \mathcal{F}(\Psi_{N,T}-\Psi) (z) + \mathcal{F}(\be') (z) \mathcal{F} (\Pi-\Pi_{N,T})(z) ) \frac{\overline{\mathcal{F}(\Pi_{N,T}) (z)} }{|\mathcal{F}(\Pi_{N,T}) (z)|^2} \1_{\{|\mathcal{F}(\Pi_{N,T})(z)|>\om\}},
	\end{align*}
	where by the Kantorovich--Rubinstein dual formulation, we have
	\bee
	|\mathcal{F}(\Pi-\Pi_{N,T}) (z)| \le |z| W_1 (\Pi_{N,T},\Pi).
	\eee
	As a result,
	\begin{align*}
	 \E \left[\int_{\R} | \be'_{N,T}(y)- \be^\ast_{N,T} (y) |^2 d y \right]^{1/2} &\le \om^{-1}   \E \left[  \int_{\R} | \Psi_{N,T} (y) - \Psi (y)|^2 d y \right]^{1/2}\\ 
	&\qquad+ \om^{-1} \Big( \E [W_1^2 (\Pi_{N,T}, \Pi)] \int_{\R} |\be'' (y)|^2 d y \Big)^{1/2}.
	\end{align*}
	Furthermore, it holds that
	\begin{align*}
	&\E \left[ \int_{\R} |\be^*_{N,T} (y) - \be' (y) |^2 d y \right]\\ 
	&\qquad= \frac{1}{2 \pi} \Big( \int_{|\mathcal{F}(\pi) (z)| > 2\om} + \int_{|\mathcal{F}(\pi) (z)| \le 2 \om} \Big) |\mathcal{F}(\be') (z)|^2 \P ( |\mathcal{F}(\Pi_{N,T}) (z)| \le \om ) d z.
	\end{align*}
	If $\om<|\mathcal{F}(\pi) (z)|/2$, then we have by the Markov and Jensen's inequalities
	\begin{align*}
	\P ( |\mathcal{F}(\Pi_{N,T}) (z) | \le \om ) &\le \P (|\mathcal{F}(\Pi_{N,T}-\Pi) (z)| \ge |\mathcal{F}(\pi) (z)| - \om  )\\ 
	&\le \frac{\E[ |\mathcal{F}(\Pi_{N,T}-\Pi) (z)|^2]}{(|\mathcal{F}(\pi)(z)| - \om)^2}.
	\end{align*}
	Consequently, we obtain
	\begin{align*}
	\E \left[ \int_{\R} |\beta^*_{N,T} (y) - \beta' (y) |^2 d y \right] &\le  \frac{1}{2 \pi } \Big( \int_{|\mathcal{F}(\pi) (z)| > 2\om}  |\mathcal{F}(\beta') (z)|^2 \frac{\E [W_1^2(\Pi_{N,T},\Pi)]|z|^2 }{|\mathcal{F}(\pi) (z)|^2/4} d z\\ 
	&\qquad
	+ \int_{|\mathcal{F}(\pi) (z)| \le 2 \om}  |\mathcal{F}(\be') (z)|^2 d z\Big),
	\end{align*}
	which completes the proof of Proposition \ref{prop3}. 

\subsection{Proof of Proposition \ref{prop:rate-entire}}
		We use Proposition \ref{prop2}. It suffices to show that on the r.h.s.\ of \eqref{int:betaMISE} the first term dominates. For this purpose, we decompose each of the last two terms into two integrals over $|z| \le \vt$ and $|z| > \vt$, respectively. For all $|z| \le \vt$, it holds $|\mathcal{F}(\pi) (z)| \ge 1/2$. Indeed, since $\pi$ is an even density, we have $\mathcal{F}(\pi) (z) = 1 + \int (\exp (\mathrm{i} zy)-1 - \mathrm{i} zy) \pi (y) d y$. Using $|\exp (\mathrm{i} x)-1- \mathrm{i} x| \le |x|^2/2$, $x \in \R$, we get
		$|\mathcal{F}(\pi) (z)| \ge 1-  m_2 z^2/2$. By Lemma \ref{mom}, we have 
		$m_2 \le 1/\la \le 1/\vt^2$.
		On the other hand, the Paley--Wiener theorem implies that $\mathcal{F}(\be') (z)$ vanishes for $|z| > \vt$. Finally, it remains
		\bee
		\E\left[\int_{\R} | \be'_{N,T}(y)-\be' (y) |^2 d y \right]  \lesssim \om^{-2}  \Big(  \E\left[\int_{\R} | \Psi_{N,T}(y)-\Psi(y) |^2 d y\right]+ 1/N_T \Big) +\vt^2/N_T,
		\eee
		where Corollary \ref{cor1} provides an upper bound on the dominating term.

\subsection{Proof of Theorem \ref{minimax}}
We will use the classical two hypotheses approach presented in the monograph \cite{T09}. 
More specifically,
we will find functions $\be_0,\be_1 \in \f_{p, C, C_0, \boldsymbol{a},\la}$
such that 
\[
\|\be'_0-\be'_1\|_{L^2(\R)}^2 = \operatorname{const} \cdot (\log N)^{-p/J} \quad \text{and} \quad
K (\Pi_{\beta_1}^{\otimes N}, \Pi_{\beta_0}^{\otimes N})  \lesssim 1,
\]
where $K (\Pi_{\beta_1}^{\otimes N}, \Pi_{\beta_0}^{\otimes N})$ denotes the 
Kullback-Leibler divergence.
 
We start with some preliminary estimates.
Let $\be \in \f_{p, C, C_0, \boldsymbol{a},\la}$. Due to the inequality $(\log \pi_\beta)''(y)\leq -2a-(\beta''\star\pi_{\be})(y) \leq -\lambda$
	the  probability measure $\Pi_\beta$ satisfies the logarithmic Sobolev inequality with constant $2/\lambda$: 
	$$
	\operatorname{Ent}_{\Pi_\beta}(f^2) \le \frac{2}{\lambda} \int_{\R} |f'(y)|^2  \Pi_\beta(dy),
	$$
	where 
	$$
	\operatorname{Ent}_{\Pi_\beta} (f^2) := \int_{\R} f^2(y) \log f^2(y)  \Pi_\beta(dy) - \int_{\R} f^2(y) \Pi_\beta(dy) \log \Big(\int_{\R} f^2(y) \Pi_\beta(dy) \Big), 
	$$
	for every smooth function $f : \R \to \R$ with $\int_{\R} |f'(y)|^2  \Pi_\beta(dy)<\infty$. Hence, for any $\be_0,\be_1 \in \f_{p, C, C_0, \boldsymbol{a},\la}$, we can bound the Kullback-Leibler divergence as
	\bee \label{KLineq}
	K(\Pi_{\be_1}, \Pi_{\be_0}) := \int_{\R} \pi_{\be_1}(y) \log \frac{\pi_{\be_1}(y)}{\pi_{\be_0}(y)} dy \leq \frac{1}{2\la}  \int_{\R} \pi_{\be_1}(y)|g(y)|^2 dy
	\eee 
	with a function
	\bee
	g = \frac{\pi'_{\be_1}}{\pi_{\be_1}} - \frac{\pi'_{\be_0}}{\pi_{\be_0}}= \Big( \Big( \sum_{1 \le j \le J} a_j \vp_j  + \be_0 \Big)' \star \pi_{\be_0}  \Big) - \Big( \Big( \sum_{1 \le j \le J} a_j \vp_j  + \be_1 \Big)' \star \pi_{\be_1} \Big).
	\eee
	We further decompose it as $g = \sum_{1<j\le J} a_j g_j + g_1 + g_0$ with
	\bee\label{def:g}
	g_j :=  \vp'_j \star (\pi_{\be_0}- \pi_{\be_1}), \ 1 < j \le J, \qquad g_1 := \be'_1 \star (\pi_{\be_0}- \pi_{\be_1}), \qquad g_0 := (\be'_{0} - \be'_{1}) \star \pi_{\be_0}.
	\eee
	In general, it seems hard to assess the functions $g_j$ for $j\geq 1$ directly (in contrast to $g_0$). Instead, we will show  that
	\bee \label{usefulin}
	\frac{1}{2\la}  \int_{\R} \pi_{\be_1}(y)|g(y)-g_0(y)|^2 dy \leq \ga^2 K(\Pi_{\be_1}, \Pi_{\be_0})
	\eee
	for some $\ga\in (0,1)$, and as a consequence of the inequality \eqref{KLineq} we deduce
	that
	\[
	K(\Pi_{\be_1}, \Pi_{\be_0}) \lesssim \int_{\R} \pi_{\be_1}(y)|g_0(y)|^2 dy.
	\]
	The latter bound will be estimated directly for a proper choice of functions $\be_0,\be_1$.
	
	We proceed with showing \eqref{usefulin}. We will find a constant $\ga_j > 0$ such that
	\bee\label{ineq:intg}
	\frac{1}{2\la} \int_{\R} \pi_{\be_1} (y) |g_j(y)|^2 dy \le \ga^2_j K(\Pi_{\beta_1},\Pi_{\beta_0}), \qquad1 \le j \le J.
	\eee
	Since
	$\|g_1\|_\infty  \le \| \be'_1 \|_\infty \|\pi_{\be_0}-\pi_{\be_1} \|_{L^1(\R)}$ and $\|\pi_{\be_0}-\pi_{\be_1} \|_{L^1(\R)} \le \sqrt{2 K(\Pi_{\be_1},\Pi_{\be_0})}$,
	we have 
	\bee 
	\ga^2_1 = \frac{\|\be'_1\|_\infty^2}{\la}.
	\eee 
	For $1 < j \le J$, we have
	\bee 
	|g_j (y)|^2 \le 
	(2j)^2 \int_{\R} (y-x)^{2(2j-1)} \Big( \sqrt{\pi_{\be_1}(x)}+\sqrt{\pi_{\be_0}(x)} \Big)^2 d x \cdot 
	\| \sqrt{\pi_{\be_1}} - \sqrt{\pi_{\be_0}} \|_{L^2(\R)}^2
	\eee
	where $\|\sqrt{\pi_{\be_1}} - \sqrt{\pi_{\be_0}} \|_{L^2(\R)}^2 \le K(\Pi_{\be_1}, \Pi_{\be_0})$ and 
	$( \sqrt{\pi_{\be_1}} +\sqrt{\pi_{\be_0}})^2 \le 2( \pi_{\be_1} + \pi_{\be_0})$. Futhermore, 
	\bee
	\int_{\R} \pi_{\be_1} (y) \int_{\R} (y-x)^{2k} \pi_{\be_i} (x) dx dy
	= \sum_{j=0}^{k} \binom{2k}{2j} m_{2j}^{\pi_{\be_1}} m_{2(k-j)}^{\pi_{\be_i}} \le \frac{C_k}{\la^k},
	\eee
	by Lemma \ref{mom} with $m_j^{\pi_{\be_i}} = \int_{\R} y^j \pi_{\be_i} (y) dy$, $i=0,1$, and
	\bee
	C_k 
	= \sum_{j=0}^k \binom{2k}{2j} (2j - 1)!! (2(k-j) - 1)!!
	= \frac{(2k)!}{k! 2^k} \sum_{j=0}^k \binom{k}{j} = \frac{(2k)!}{k!}.
	\eee
	Hence, \eqref{ineq:intg} holds true with
	\bee
	\ga_j^2 = \frac{2 (2j)^2 C_{2j-1}}{\la^{2j}}
	\qquad 1 < j \le J.
	\eee 
	We conclude that 
	\begin{align*}
	(K (\Pi_{\beta_1}, \Pi_{\beta_0}))^{1/2} \le \Big(\sum_{1<j \le J} a_j \ga_j + \ga_1 \Big) (K (\Pi_{\be_1}, \Pi_{\be_0}))^{1/2} +  \Big( \frac{1}{2\la}\int_{\R} |g_0 (y)|^2 \pi_{\beta_1}(y) dy\Big)^{1/2}.
	\end{align*}
	Consequently, we deduce the inequality 
	\bee \label{firstine}
	K (\Pi_{\beta_1}, \Pi_{\beta_0}) \le \frac{1}{2\la(1 - \sum_{1<j \le J} a_j \ga_j - \ga_1)^2} \int_{\R} |g_0 (y)|^2 \pi_{\beta_1}(y)d y.
	\eee
	Due to \eqref{firstine} we only need to handle the last term in \eqref{def:g}. For this purpose we use following construction: We introduce the constants $\rho>0$, $M>0$ ($\rho\to 0$ and $M\to \infty$ to be chosen later) and a function $\phi\in C^{2}(\R) $ with 
	\[
	\text{supp}(\phi)=[-2,-1] \cup [1,2],
	\]
	and set
	\bee
	\be_0(y)=f_0(y), \qquad \be_1(y)=f_0(y) + \rho M \phi(y/M). 
	\eee
	Here $f_0 \in \f_{p, C/4, C_0/2, (a_1/2, \dots, a_J/2^J),\la/2} \subset \f_{p, C, C_0, \boldsymbol{a},\la}$. 
	To ensure
	that $\be_1 \in \f_{p, C, C_0, \boldsymbol{a},\la}$ we assume
	\begin{align}\label{condi1}
	\rho M\|\phi\|_{\infty} \le C_0/2, \quad &\rho \|\phi'\|_{\infty} \le C_1- C_1/2^{1/2}, \quad (\rho/M) \|\phi''\|_{\infty} \le C_2/2,\\
	\label{condi2} &\rho^2 M^{2p+1} \int_{\R} y^{2p} |\phi'(y)|^2 dy \le C/4.
	\end{align}
	In particular, condition \eqref{condi2} allows us to choose and later use $\rho = c M^{-p-1/2}$ for some $c>0$.
	Furthermore, the condition $p>1/2$ is required to ensure that  $\rho M=O(1)$.
	We obviously  have that
	\bee
	\|\be'_0-\be'_1\|_{L^2(\R)}^2 =\rho^2M \|\phi'\|_{L^2(\R)}^2.
	\eee
	Next, we will bound the right hand side of \eqref{firstine}, where recall \bee
	g_0 (y) = ((\be'_1-\be'_0) \star \pi_{\be_0}) (y).
	\eee 
	For this purpose we note that $g_0(y) = - g_0 (y)$ and then decompose $\int_{\R} =2 (\int_0^k + \int_k^\infty)$, where the threshold $k<M$ will be chosen later. Since $Z_{\pi_{\be_1}} \geq \int_{\R} \exp(- \sum_{j=0}^J \al_j y^{2j} - \|\be_1\|_{\infty}) dy$ and $\| g_0 \|_\infty \le \| \be'_0 - \be'_1 \|_\infty$ we deduce that 
	\bee
	\int_k^\infty |g_0(y)|^2 \pi_{\be_1}(y) dy \lesssim \rho^2
	\int_k^\infty \exp(-a_J y^{2J}) dy  \lesssim \rho^2 \frac{\exp(-a_J k^{2J})}{k^{2J-1}}. 
	\eee
	On the other hand, we get 
	\[
	\int_0^k |g_0(y)|^2 \pi_{\be_1}(y) dy \lesssim \rho^2 k \sup_{y \in [0,k]} | (\phi'(\cdot/M) \star \pi_{\be_0})(y) |^2
	\] 
	and
	\[
	|(\phi'(\cdot/M) \star \pi_{\be_0})(y)| \lesssim \int_{-\infty}^{k-M} \pi_{\be_0}(y) dy \lesssim \frac{\exp(-a_J (M-k)^{2J})}{(M-k)^{2J-1}}.
	\]
	Consequently, choosing $k=M/2$ we deduce from  \eqref{firstine} that 
	\bee
	K (\Pi_{\beta_1}, \Pi_{\beta_0}) \lesssim \exp(-a_J (M/2)^{2J})
	\eee
	(recall that $\rho = O( M^{-p-1/2})$). Now, choosing $M= 2 ((\log N)/a_J)^{1/(2J)}$ we finally obtain that 
	\[
	K (\Pi_{\beta_1}^{\otimes N}, \Pi_{\beta_0}^{\otimes N}) = N K (\Pi_{\beta_1}, \Pi_{\beta_0}) \lesssim 1.
	\]
	On the other hand, since $\rho= c M^{-p-1/2}$ we deduce that 
	\[
	\|\be'_0-\be'_1\|_{L^2(\R)}^2 = 
	\operatorname{const} \cdot (\log N)^{-p/J},
	\]
	which by \cite[(2.9) and Theorem 2.2(iii)]{T09} completes the proof of our Theorem  \ref{minimax}.

\bibliographystyle{plain}
\bibliography{gran-part}

\begin{thebibliography}{10}

\bibitem{benachour1998nonlinear}
Sa{\"i}d Benachour, Bernard Roynette, Denis Talay, and Pierre Vallois.
\newblock Nonlinear self-stabilizing processes -- {I} {E}xistence, invariant
  probability, propagation of chaos.
\newblock {\em Stochastic Processes and their Applications}, 75(2):173--201,
  1998.

\bibitem{cattiaux2008probabilistic}
Patrick Cattiaux, Arnaud Guillin, and Florent Malrieu.
\newblock Probabilistic approach for granular media equations in the
  non-uniformly convex case.
\newblock {\em Probability Theory and Related Fields}, 140(1-2):19--40, 2008.

\bibitem{chang2011approximate}
Jinyuan Chang and Song~Xi Chen.
\newblock On the approximate maximum likelihood estimation for diffusion
  processes.
\newblock {\em The Annals of Statistics}, 39(6):2820--2851, 2011.

\bibitem{CMP20}
Gabriela Cio{\l{}}ek, Dmytro Marushkevych, and Mark Podolskij.
\newblock On {D}antzig and {L}asso estimators of the drift in a high
  dimensional {O}rnstein-{U}hlenbeck model.
\newblock {\em Electronic Journal of Statistics}, 14(2):4395--4420, 2020.

\bibitem{dellamaestra2021nonparametric}
Laetitia Della~Maestra and Marc Hoffmann.
\newblock Nonparametric estimation for interacting particle systems:
  {M}ckean-{V}lasov models.
\newblock {\em Probability Theory and Related Fields}, 2021.
\newblock https://doi.org/10.1007/s00440-021-01044-6.

\bibitem{frank2005nonlinear}
Till~Daniel Frank.
\newblock {\em Nonlinear Fokker-Planck equations: fundamentals and
  applications}.
\newblock Springer Science \& Business Media, 2005.

\bibitem{GM19}
St{\'e}phane Ga{\"i}ffas and Gustaw Matulewicz.
\newblock Sparse inference of the drift of a high- dimensional
  {O}rnstein-{U}hlenbeck process.
\newblock {\em Journal of Multivariate Analysis}, 169:1--20, 2019.

\bibitem{kolokoltsov2010nonlinear}
Vassili~N. Kolokoltsov.
\newblock {\em Nonlinear Markov processes and kinetic equations}, volume 182.
\newblock Cambridge University Press, 2010.

\bibitem{kutoyants2013statistical}
Yury~A. Kutoyants.
\newblock {\em Statistical inference for ergodic diffusion processes}.
\newblock Springer Science \& Business Media, 2013.

\bibitem{GL20}
Catherine Laredo and Valentine Genon-Catalot.
\newblock Parametric inference for small variance and long time horizon
  {M}c{K}ean-{V}lasov diffusion models, 2020.
\newblock hal-03095560.

\bibitem{malrieu2003convergence}
Florent Malrieu.
\newblock Convergence to equilibrium for granular media equations and their
  euler schemes.
\newblock {\em The Annals of Applied Probability}, 13(2):540--560, 2003.

\bibitem{mckean1966class}
Henry~P. McKean.
\newblock A class of {M}arkov processes associated with nonlinear parabolic
  equations.
\newblock {\em Proceedings of the National Academy of Sciences of the United
  States of America}, 56(6):1907--1911, 1966.

\bibitem{M09}
Alexander Meister.
\newblock {\em Deconvolution problems in nonparametric statistics}.
\newblock Lecture Notes in Statistics, 2009.

\bibitem{nickl2017nonparametric}
Richard Nickl and Jakob S{\"o}hl.
\newblock Nonparametric {B}ayesian posterior contraction rates for discretely
  observed scalar diffusions.
\newblock {\em The Annals of Statistics}, 45(4):1664--1693, 2017.

\bibitem{ren2019least}
Panpan Ren and Jiang-Lun Wu.
\newblock Least squares estimator for path-dependent {M}c{K}ean-{V}lasov {SDE}s
  via discrete-time observations.
\newblock {\em Acta Mathematica Scientia}, 39B(3):691--716, 2019.

\bibitem{strauch2018adaptive}
Claudia Strauch.
\newblock Adaptive invariant density estimation for ergodic diffusions over
  anisotropic classes.
\newblock {\em The Annals of Statistics}, 46(6B):3451--3480, 2018.

\bibitem{sznitman1991topics}
Alain-Sol Sznitman.
\newblock Topics in propagation of chaos.
\newblock In {\em Ecole d'{\'e}t{\'e} de probabilit{\'e}s de Saint-Flour
  XIX--1989}, pages 165--251. Springer, 1991.

\bibitem{T09}
Alexandre~B. Tsybakov.
\newblock {\em Introduction to nonparametric estimation}.
\newblock Springer Series in Statistics, 2009.

\end{thebibliography}

\end{document}